\documentclass{amsart}
\usepackage{graphics}
\usepackage{amscd}
\usepackage{mathdots}
\usepackage{amsfonts,amssymb,latexsym,color}
\usepackage{xypic}

\newtheorem{theorem}{Theorem}
\newtheorem{proposition}[theorem]{Proposition}
\newtheorem{corollary}[theorem]{Corollary}
\newtheorem{lemma}[theorem]{Lemma}

\theoremstyle{definition} 
\newtheorem{definition}[theorem]{Definition}
\newtheorem{remark}[theorem]{Remark}
\newtheorem{example}[theorem]{Example}
\newcommand{\cc}{\mathbb{C}}
\newcommand{\rr}{\mathbb{R}}
\newcommand{\pp}{\mathbb{P}}
\newcommand{\nn}{\mathbb{N}}
\newcommand{\qq}{\mathbb{Q}}
\newcommand{\zz}{\mathbb{z}}

\newcommand{\gl}{\text{Gl}}
\newcommand{\linebu}{\mathsf{N}}

\newcommand{\Sym}{\text{Sym}}
\newcommand{\rk}[1]{\text{rk}({#1})}

\newcommand{\rest}[2]{{#1}_{\mid_{#2}}}
\newcommand{\tipo}{\underline{\mathfrak{t}}}

\newcommand{\kk}[2]{\texttt{k}_{{{#1}{#2}}}}

\newcommand{\degpar}{\deg_{\mathtt{par}}}
\newcommand{\efil}{E^{\bullet}}
\newcommand{\alphafil}{\underline{\alpha}}
\newcommand{\filtration}{\efil,\alphafil}
\newcommand{\efascio}{\mathcal{E}}

\newcommand{\ofascio}{\mathcal{O}}

\newcommand{\ttI}{\mathtt{I}}
\newcommand{\ttJ}{\mathtt{J}}
\newcommand{\cttI}{{|\mathtt{I}|}}
\newcommand{\ttIbar}{{\overline{\mathtt{I}}}}

\newcommand{\ijbar}{{\{\overline{i,j}\}}}
\newcommand{\map}{Q}

\title[Quadric bundles]{Stability of Quadric Bundles}
\date{\today}

\author{Alessio Lo Giudice}
\address{\scriptsize Alessio Lo giudice:
Scuola Internazionale Superiore di Studi Avanzati,
Via Bonomea 265, 34136 Trieste, Italia}
\email{alessio.logiudice@sissa.it}
\author{Andrea Pustetto}
\address{\scriptsize Andrea Pustetto:
Scuola Internazionale Superiore di Studi Avanzati,
Via Bonomea 265, 34136 Trieste, Italia}
\email{andrea.pustetto@sissa.it}

\begin{document}

\begin{abstract}
A decorated vector bundle on a smooth projective curve $X$ is a pair $(E,\varphi)$ consisting of a vector bundle and a morphism $\varphi:(E^{\otimes a})^{\oplus b}\to(\det E)^{\otimes c}\otimes\linebu$, where $\linebu\in\text{Pic}(X)$. There is a suitable semistability condition for such objects which has to be checked for any weighted filtration of $E$. We prove, at least when $a=2$, that it is enough to consider filtrations of length $\leq2$. In this case decorated bundles are very close to quadric bundles and to check semistability condition one can just consider the former. A similar result for L-twisted bundles and quadric bundles was already proved (\cite{GPGMR}, \cite{Sch}). Our proof provides an explicit algorithm which requires a destabilizing filtration and ensures a destabilizing subfiltration of length at most two. Quadric bundles can be thought as a generalization of orthogonal bundles. We show that the simplified semistability condition for decorated bundles coincides with the usual semistability condition for orthogonal bundles. Finally we note that our proof can be easily generalized to decorated vector bundles on nodal curves.
\end{abstract}

\maketitle

\section{Introduction}

Let $X$ be a smooth projective curve of genus $g\geq1$. We are interested in studying semistability conditions for certain decorated vector bundles, i.e., pairs $(E,\varphi)$ where $E$ is a vector bundle on $X$, while $\varphi$ is a morphism between $(E\otimes E)^{\oplus b} \to (\det E)^{\otimes c}\otimes\linebu$; here $b,c$ are natural numbers and $\linebu$ is a (fixed) line bundle on $X$. A semistability condition for these objects was introduced by Schmitt \cite{Sch}; it has to be checked for any weighted filtration of the underlying vector bundle. In this paper we show that one can just consider filtrations of length two and subbundles (see Theorem \ref{teo-main}). Moreover, thanks to a certain ``symmetry'' of the semistability condition for decorated vector bundles, we show that, without loss of generality, we can restrict our attention to pairs $(E,\map)$, called \textit{quadric bundles}, where $E$ is as before, while $\map\colon\Sym^2 E\to\linebu$ is a nonzero morphism of vector bundles. In rank $3$ 
case these objects were called \textit{conic bundles} and were studied by G\'omez and Sols in \cite{GS}. In \cite{GPGMR}, Garc\'ia-Prada, Gothen and Mundet i Riera studied more general objects than quadric bundles (but distinct from decorated vector bundles), called L-twisted $\text{Sp}(2n,\rr)$-Higgs pairs. They gave a semistability condition and prove a result similar to Corollary \ref{cor-equiv-ss-conditions} for these L-twisted bundles by using a result analogous to our Theorem \ref{teo-sub-critical} that was proved by Schmitt (cf. Theorem 2.8.4.13 of \cite{Sch}). We provide an algorithmic proof of these two results and give an explicit example (Example \ref{example-teo-sub-critical}). In Section \ref{section-nodal} we show that this proof can be easily generalized to decorated bundles on nodal curves using the theory of parabolic bundles.\\

Moreover we verify that, for rank equal to $3$, one obtains the semistability condition given by G\'omez and Sols in \cite{GS}. Finally, when $\linebu=\ofascio_X$ and $\map$ induces an isomorphism between $E$ and its dual $E^\vee$, the quadric bundle is an orthogonal bundle. In this case there is a classical definition of semistability by Ramanan \cite{Ramanan}; in Section \ref{orthogonalbundles}, we check that it coincides with our definition.

Whenever the word '(semi)stable' appears in a statement with the symbol '$(\leq)$', two statements should be read; the first with the word `stable' and strict inequality, and the second with the word `semistable' and the relation `$\leq$'.

\section{Decorated vector bundles}
Decorated vector bundles were studied by Schmitt in \cite{Sch}. They are a versatile tool for studying bundles or sheaves over varieties. A decorated vector bundle over a projective variety $X$ consists of a pair $(E,\varphi)$, where $E$ is a vector bundle over $X$, and $\varphi$ is a morphism of vector bundles
\begin{equation*}
 \varphi\colon (E^{\otimes a})^{\oplus b}\to (\det E)^{\otimes c}\otimes\linebu,
\end{equation*}
with $\linebu\in\text{Pic}(X)$ and $a,b,c\in\nn$. In this case we will say that $(E,\varphi)$ is a decorated vector bundle of \textit{type} $(a,b,c,\linebu)$.\\
These objects are important because they provide a common generalization of several notions of bundles with structure; indeed principal bundles, principal Higgs bundles, vector bundles, Bradlow pairs, quadric bundles, orthogonal bundles, etc., can be regarded as decorated vector bundles. More generally, if $\rho:\gl(V)\to\gl(W)$ is a \textit{homogeneous} representation, i.e., if
\begin{equation*}
 \rho(z\cdot id_{\gl(V)})= z^{h}\cdot id_{\gl(W)}
\end{equation*}
for some $h\in\zz$, and if $E$ is a $\gl(V)$-bundle, we can consider the vector bundle $E_\rho=E\times_{\rho}\gl(W)$ as a direct summand of the vector bundle $(E^{\otimes a})^{\oplus b}\otimes (\det E)^{-\otimes c}$ (\cite{Sch} Corollary 1.1.5.4).\\
Useful as they are, decorated bundles have however a drawback: their semistability condition has to be checked for all weighted filtrations $(\filtration)$ of $E$ and, in general, it is not easy to calculate.\\
The main result of this paper is a simpler statement of the semistability condition for decorated vector bundles having $a=2$ ($b$ and $c$ are left arbitrary), i.e., the following theorem:

\begin{theorem}\label{teo-main}
Let $(E,\varphi)$ be a decorated vector bundle of type $(2,b,c,\linebu)$. $(E,\varphi)$ is (semi)stable if and only if $P(\efil,\underline{1})+\delta\,\mu(\efil,\underline{1};\map)(\geq)0$ for any filtration of length $\leq2$.
\end{theorem}

We recall the definition of semistable decorated vector bundle given by Schmitt.\\

Let $(E,\varphi)$ be a decorated vector bundle of type $(a,b,c,\linebu)$ and let $d=\deg(E),r=\rk{E}$ be integers.\\
Let $(\filtration)$ be a weighted filtration of $E$ indexed by $\ttI\doteqdot\{1,\dots,t\}$, i.e., the datum of a filtration $0\subset E_{1}\subset\dots\subset E_{t}\subset E_r=E$ and a weight vector $\alphafil=(\alpha_{1},\dots,\alpha_{t})$. Define $\ttIbar=\ttI\cup\{r\}$, let
\begin{equation*}
\gamma_\ttI\doteqdot\sum_{i\in\ttI} \alpha_i (\underbrace{\rk{E_i}-r,\dots,\rk{E_i}-r}_{\rk{E_i}\text{-times}},\underbrace{\rk{E_i},\dots,\rk{E_i}}_{r-\rk{E_i}\text{-times}}).
\end{equation*}
Finally define
\begin{equation*}
\mu(\filtration;\map)\doteqdot -\min_{i_1,\dots,i_a\in\ttIbar}\{\gamma_\ttI^{(i_1)}+\dots+\gamma_\ttI^{(i_a)}\; | \; \rest{\varphi}{(E_{i_1}\otimes\dots\otimes E_{i_a})^{\oplus b}}\not\equiv0\}.
\end{equation*}

\begin{definition}[\textbf{Semistability}]
We say that $(E,\varphi)$ is \textit{(semi)stable} if for any weighted filtration $(\filtration)$ of $E$ the following inequality holds:
\begin{equation*}
P(\filtration)+\delta\,\mu(\filtration;\map)(\geq)0
\end{equation*}
where
\begin{equation*}
P(\filtration)\doteqdot\sum_{i\in\ttI}\alpha_i(\deg(E)\rk{E_i}-\rk{E}\deg(E_i)).
\end{equation*}
\end{definition}

Now let $\efil$ be a filtration of $E$ indexed by $\ttI$. We can construct a tensor $M_{\ttI}(\efil,\varphi)=(m_{i_1\dots i_a})_{i_1,\dots,i_a\in\ttIbar}$, associated to the given filtration and to the decoration morphism $\varphi$, as follows:
\begin{equation*}
m_{i_1\dots i_a}\doteqdot\begin{cases}
                          1 \text{ if } \rest{\varphi}{(E_{\sigma(i_1)}\otimes\dots\otimes E_{\sigma(i_a)})^{\oplus b}}\not\equiv0 \text{ for any permutation }\sigma\\
                          0 \text{ otherwise.}
                         \end{cases}
\end{equation*}
Note that
\begin{itemize}
 \item[-] $M_{\ttI}(\efil,\varphi)$ is symmetric, i.e. $m_{i_1\dots i_a}=m_{\sigma(i_1)\dots \sigma(i_a)}$ for any permutation $\sigma$ of $a$-terms;
\item[-] $\mu(\filtration;\varphi)=-\min_{i_1,\dots,i_a\in\ttIbar}\{\gamma_\ttI^{(i_1)}+\dots+\gamma_\ttI^{(i_a)} \, | \, m_{i_1\dots i_a}\neq0 \}$.
\end{itemize}

Therefore, in order to calculate $\mu(\filtration;\varphi)$, it is enough to know the vector $\gamma_\ttI$ and the tensor $M_{\ttI}(\efil,\varphi)$, i.e. is enough to know if $\varphi$ vanish or not. For this reason the semistability condition does not distinguish between $\varphi$ and $\varphi'$ if for any filtration $\efil$ we have that $M_{\ttI}(\efil,\varphi)=M_{\ttI}(\efil,\varphi')$.\\

Let $a=2$. In order to prove Theorem \ref{teo-main}, we can suppose, without loss of generality, that the morphism $\varphi$ is symmetric, i.e. we will consider pairs $(E,\varphi)$ such that exists a morphism $\map$ making the following diagram commutative:
\begin{equation*}
 \xymatrix{ E\otimes E \ar[d] \ar[r]^{\varphi} & \linebu\\
            \Sym^2(E). \ar[ur]_Q & }
\end{equation*}
Such pairs, that from now on we will denote by $(E,\map)$, are called \textit{quadric bundles}. Therefore, thanks to the previous considerations, we can prove, without loss of generality, the main result in the particular case of quadric bundles. Moreover any result in the following section holds also for decorated vector bundles of type $(2,b,\linebu)$.\\

\section{Quadric bundles}\label{section-quadric}
Let $X$ be a smooth projective complex curve of genus $g$ and $\linebu$ a line bundle over $X$. Let us fix integers $r>0$ and $d$. 

\begin{definition}[\textbf{Quadric Bundles}]
A quadric bundle on $X$ of type $(r,d,\linebu)$ is a pair $(E,\map)$ where $E$ is a vector bundle of rank $r$ and degree $d$ on $X$, and
\begin{equation*}
\map: Sym^2 E \to \linebu,
\end{equation*}
is a morphism of vector bundles.\\
A morphism between quadric bundles $f:(E,\map)\to (E',\map')$ is a morphism $f:E \to E'$ of vector bundles such that there is a commutative diagram 
\begin{equation*}
\CD
\Sym^2 E @>\Sym^2 f>> \Sym^2 E' \\
@V{\map}VV             @V{\map'}VV \\
\linebu @>\lambda>>   \linebu
\endCD
\end{equation*}
where $\lambda$ is a scalar multiple of the identity.
\end{definition}

The term ``quadric'' comes from the fact that for every $x\in X$ the morphism $\map$ restricted to the fibre $E_x$ defines a bilinear symmetric form and so a quadric in $\pp^{r-1}$.\\

If the morphism $\map$ is the zero morphism, a quadric bundle is just an ordinary vector bundle, so from now on we suppose that $\map$ is not identically zero. Note that even if the map $\map$ is non-zero it could happen that restricted to a subbundle it vanishes.\\

\begin{remark}[\textbf{Notation}]
 For convenience's sake we introduce the following notation: if $(\filtration)$ is a weighted filtration as before, indexed by $\ttI=\{i_1,\dots,i_t\}$, we define $\ttIbar=\ttI\cup\{r\}$, where it is always understood that $E_r=E$. Given a filtration $(\filtration)$ indexed by $\ttI$, we will denote with $\mu_\ttI$ the number $\mu(\filtration;\map)$ if is clear from the context which filtration, weights and morphism we are considering. Moreover if $\efil$ is a filtration indexed by $\ttI$ ($F$ is a subbundle of $E$) then we denote by $r_i$ and $d_i$ ($r_F$ and $d_F$) the rank and the degree of $E_i$ for any $i\in\ttI$ ($F$ respectively). Finally, if we write ``filtration'' instead of ``weighted filtration'', we mean that all weights are equal to one.
\end{remark}

Let $(\filtration)$ be a weighted filtration indexed by $\ttI\doteqdot\{i_1,\dots,i_t\}$. We define, as before, the following vector of length $r$:
\begin{equation*}
\gamma_\ttI\doteqdot\sum_{i\in\ttI} \alpha_i (\underbrace{r_i-r,\dots,r_i-r}_{r_i\text{-times}},\underbrace{r_i,\dots,r_i}_{r-r_i\text{-times}}).
\end{equation*}
If $E_1$ and $E_2$ are subbundles of $E$, we denote by $E_1 E_2$ the subbundle of $\Sym^2 E$ generated by elements of the form $e_1 e_2$ where $e_1$ and $e_2$ are local sections of $E_1$ and $E_2$; with this convention we define
\begin{equation*}
\mu(\filtration;\map)\doteqdot -\min_{i,j\in\ttIbar}\{\gamma_\ttI^{(i)}+\gamma_\ttI^{(j)}\; | \; \rest{\map}{E_{i} E_{j}}\not\equiv0\}
\end{equation*}

\begin{definition}[\textbf{$\delta$-Semistable Quadric Bundles}]\label{def-semistability}
Fix $\delta\in\qq_{>0}$. Let $(E,\map)$ be a quadric bundle of type $\tipo=(r,d,\linebu)$. $(E,\map)$ is $\delta$-(semi)stable if and only if for any weighted filtration $(\filtration)$, indexed by $\ttI$, the following inequality holds:
\begin{equation*}
P(\filtration)+\delta\,\mu(\filtration;\map)(\geq)0
\end{equation*}
where, as usual,
\begin{equation*}
P(\filtration)\doteqdot\sum_{i\in\ttI}\alpha_i(d r_i-r d_i).
\end{equation*}
Sometimes we use the notation $P_\ttI$ instead of $P(\filtration)$ when is clear from the context which weighted filtration indexed by $\ttI$ we are considering.\\
\end{definition}

Observe that, if $l,m\in\ttIbar$ is a pair of indexes which realizes the minimum for the set $\{\gamma_\ttI^{(i)}+\gamma_\ttI^{(j)} \,|\, \rest{\map}{E_{i} E_{j}}\neq0\}_{i,j\in\ttIbar}$, then defining:
\begin{equation*}
c_k\doteqdot r_k d-d_k r - 2\delta r_k \qquad\text{ and }\qquad R_\ttI(l)\doteqdot\begin{cases}
                                                                                   \sum_{i\in I,i\geq l}\alpha_i \qquad\text{if }l<r\\
                                                                                   0 \qquad\qquad\qquad\text{if }l=r
                                                                                  \end{cases}
\end{equation*}
we have that
\begin{align*}
 P_\ttI+\delta\mu_\ttI=P_\ttI+\delta(\gamma_\ttI^{(l)}+\gamma_\ttI^{(m)}) & = P_\ttI + \delta \left( \sum_{s\in\ttI} \alpha_s r_s - \sum_{s\geq l}r\alpha_s + \sum_{s\in\ttI} \alpha_s r_s - \sum_{s\geq m}r\alpha_s\right)\\
& = \sum_{k\in\ttI} \alpha_k\,c_k+r\delta(R_\ttI(l)+R_\ttI(m)).
\end{align*}

\begin{definition}
 If $F$ is a subbundle of $E$ we define a function $\kk{F}{E}$ as follows:
\begin{equation*}
\kk{F}{E}=\left\{
\begin{array}{lcl}
2, & \text{if} & \rest{\map}{F F}\neq 0\\
1, & \text{if} & \rest{\map}{F E}\neq 0 = \rest{\map}{F F} \\
0, & \text{if} & \rest{\map}{F E}=0.
\end{array}
\right.
\end{equation*}
\end{definition}

Note that if $(E,\map)$ is a $\delta$-(semi)stable quadric bundle then for any proper subbundle $F\subset E$ one has
\begin{equation}\label{eq-k-semistability}
 \mu(F)-\delta\frac{\kk{F}{E}}{\rk{F}}(\leq)\mu(E)-\delta\frac{2}{\rk{E}};
\end{equation}
this follows directly from the equality
\begin{equation*}
\mu(0\subset F\subset E,1;\map)=\rk{E}\,\kk{F}{E}-2\,\rk{F}.
\end{equation*}

The converse is not true in general because $\mu(\filtration;\map)$ is not additive for all filtrations, i.e., is not always true that
\begin{equation}\label{eq-mu-additivity}
 \mu(\filtration;\map)=\sum_{i\in\ttI}\mu(0\subset E_i\subset E,\alpha_i;\map).
\end{equation}
We will call \textbf{non-critical} a filtration for which \eqref{eq-mu-additivity} holds, \textbf{critical} otherwise. Therefore checking semistability condition over non-critical filtrations is the same to check it over subbundles.

\begin{definition}[\textbf{$\kk{}{}$-(semi)stability}]
 We say that $(E,\map)$ is $\kk{}{}$-(semi)stable if and only if for any proper subbundle $F$ the inequality \eqref{eq-k-semistability} holds.
\end{definition}

From previous considerations it is easy to understand that the following conditions are equivalent:
\begin{enumerate}
 \item[i)] $(E,\map)$ is $\delta$-(semistable);
 \item[ii)] for any subbundle $F$ and for any critical weighted filtration $(\filtration)$ the following inequalities hold
 \begin{equation*}
 (r_F d-r d_F)+\delta(r\kk{F}{E}-2r_F)(\geq)0, \qquad P(\filtration)+\delta\,\mu(\filtration;\map)(\geq)0.
 \end{equation*}
\end{enumerate}

Observe that the first part of condition $(ii)$ amounts just to requiring that $(E,\map)$ is $\kk{}{}$-(semi)stable.\\

\begin{remark}
Let $0\subset E_i\subset E_j\subset E$ be a critical filtration of length $2$, then the fact that the filtration is critical has a geometrical interpretation, i.e. for a generic point $x\in X$ we obtain a flag of the vector space $E_x\simeq\cc^r$, the morphism $\map$ gives us a quadric $C_x$ in $\pp^{r-1}$ moreover the fact that the filtration is critical tells us that the $i-1$ dimensional space $\pp(E_{i,x})$ lies in $C_x$ and the $j-1$ space $\pp(E_{j,x})$ belongs to the tangent space $T_p C_x$ for any $p\in\pp(E_{i,x})$.
\end{remark}

Given a filtration $\efil$ indexed by $\ttI=\{1,\dots,s\}$ ($s\leq r-1$) using local sections we can construct the (symmetric) matrix $M_{\ttI}(\efil,\map)=(m_{ij})_{ij\in\ttIbar}$ where
\begin{equation*}
m_{ij}=\begin{cases}
        1 \text{ if } \rest{\map}{E_i E_j}\neq 0\\
        0 \text{ if } \rest{\map}{E_i E_j}=0.
       \end{cases}
\end{equation*}

Since there are many critical filtrations, this semistability condition is quite hard to check in general, so we would like to restrict our attention just to subbundles and a special subset of critical filtrations.\\

\begin{proposition}\label{prop-mu-critical-filtration}
 Let $(0\subset E_i\subset E_j\subset E \,,\, \alpha_i,\alpha_j)$ be a critical weighted filtration of length two, then
\begin{equation*}
 -\mu_{\{i,j\}}=\alpha_i r_i+\alpha_j r_j-r \max\{\alpha_i+\alpha_j,2\alpha_j\}.
\end{equation*}
\end{proposition}
\begin{proof}
 We consider the matrix $M_{\{i,j\}}=(m_{lk})_{l,k\in\ijbar}$ representing $\map$ with respect to the given filtration. The only critical case is the following
\begin{equation*}
 \begin{pmatrix}
0 & 0 & 1\\
0 & 1 & 1\\
1 & 1 & 1
\end{pmatrix},
\end{equation*}
and in this case 
\begin{equation*}
-\mu_{\{i,j\}}=\begin{cases}
         \gamma_{\{i,j\}}^{(i)}+\gamma_{\{i,j\}}^{(r)} \text{ if }\alpha_i\geq\alpha_j=\alpha_ir_i+\alpha_jr_j-r(\alpha_i+\alpha_j)\\
         \phantom.\\
         \gamma_{\{i,j\}}^{(j)}+\gamma_{\{i,j\}}^{(j)} \text{ if }\alpha_i\leq\alpha_j=\alpha_ir_i+\alpha_jr_j-2r\alpha_j
         \end{cases}
\end{equation*}
and this finishes the proof.
\end{proof}

\begin{remark}
It easy to see that for any filtration one has $\mu(\efil;\map)\leq\sum\mu(0\subset E_i\subset E;\map)$. Therefore any subfiltration of a non-critical filtration is still non-critical; indeed if $\efil_{\ttI}$ is a non-critical filtration indexed by $\ttI$ and $\ttJ\subset\ttI$ indexes a subfiltration of $\efil$ then if $\mu(\efil_{\ttJ};\map)<\sum_{i\in\ttJ}\mu(0\subset E_i\subset E;\map)$ we should have that $\mu(\efil_{\ttI\smallsetminus\ttJ};\map)>\sum_{i\in\ttI\smallsetminus\ttJ}\mu(0\subset E_i\subset E;\map)$, which is absurd. Theorem \ref{teo-sub-critical} gives a partial inverse of this remark.
\end{remark}

\begin{theorem}\label{teo-sub-critical}
 Let $(E,\map)$ be a quadric bundle. It is enough to check semistability condition on subbundles and critical weighted filtrations of length two.
\end{theorem}
\begin{proof}
 We will prove that, given a weighted filtration $(\filtration)$ indexed by $\ttI$ with $\cttI\geq3$ there exist two weighted subfiltrations (if necessary with different weights) such that
\begin{equation}\label{eq-teo-sub-critical}
 P_\ttI+\delta\mu_\ttI = P_{\ttJ}+\delta\mu_{\ttJ} + P_{K}+\delta\mu_{K}
\end{equation}
Without loss of generality we can assume that the set $\ttI$ is well ordered and that the first (last) element of the set is indexed by $1$ ($r-1$ respectively), moreover if $i\in\ttJ$ (where $\ttJ$ is any well ordered set of indexes) with the notation $i+1$ we mean the successor of $i$ inside $\ttJ$. Let $M=(m_{ij})_{ij\in\ttI}$ be the matrix representing the morphism $\map$ with respect to the given filtration $(\filtration)$. If the first row of the matrix is zero then we can split the filtration as $(E_1,\alpha_1)$ and $(\filtration)_{\ttI\smallsetminus\{1\}}$ and so we are done. Therefore we can assume that this is not the case.\\
Let us denote $j_i$, for any $i\in\ttI$, the minimum of the set $\{j\in\ttI \,|\, m_{ij}\neq0 \}$. We will distinguish the case in which $j_1<r$ or $j_1=r$. In the former case we split the filtration in two subfiltrations: $(\filtration)_{\ttI\smallsetminus\{r-1\}}$ and $(E_{\{r-1\}},\alpha_{\{r-1\}})$. With these choices equality \eqref{eq-teo-sub-critical} holds, in fact, denoting with
\begin{equation*}
 R_\ttI\doteqdot\max_{s,t\in\ttI}\{ R_\ttI(s)+R_\ttI(t) \,|\, \rest{\map}{E_sE_t}\neq 0\} = \max_{i\in\ttI}\{ R_\ttI(i)+R_\ttI(j_i) \}
\end{equation*}
and calling $k$ the index such that $(k,j_k)$ realizes the maximum, one has that
\begin{align*}
  R_\ttI & = R_\ttI(k)+ R_\ttI(j_{k})=\alpha_{k}+\dots+\alpha_{j_{k}-1}+2\alpha_{j_{k}}+\dots+2\alpha_{r-1}\\
         & = R_{\ttI\smallsetminus\{r-1\}}+2\alpha_{r-1}\\
         & = R_{\ttI\smallsetminus\{r-1\}}+R_{\{r-1\}}
\end{align*}
where the second equality holds since the maximum in the subfiltration indexed by $\ttI\smallsetminus\{r-1\}$ is still achieved in $(k,j_{k})$ if $j_k\neq r-1$, otherwise in $(1,r)$ if $j_k=r-1$. The last equality holds from the assumption $j_1\neq r$ which implies $m_{1 r-1}\neq 0$ and consequently $m_{r-1 r-1}\neq 0$. Finally, recalling that $P_\ttI+\delta\mu_\ttI=\sum_{s\in\ttI} \alpha_s c_s + \delta r R_\ttI$, we get the thesis.\\
Suppose now that $j_1=r$ and that $\max_{i\in\ttI}\{ R_\ttI(i)+R_\ttI(j_i) \}$ is gained in $k$. Therefore, for all $s\in\ttI$ such that $j_s\neq j_{s-1}$ we have a set of inequalities in the variables $\alpha_i$ given from the inequalities $R_\ttI(k)+R_\ttI(j_k)\geq R_\ttI(s)+R_\ttI(j_s)$, i.e.
\begin{align}
 \alpha_s+\dots+\alpha_{k-1}\leq \alpha_{j_k}+\dots+\alpha_{j_s - 1} \qquad & s\leq k-1 \label{ineq-teo-sub-criticaluno}\\
 \alpha_k+\dots+\alpha_{s-1}\geq \alpha_{j_s}+\dots+\alpha_{j_k - 1} \qquad & s\geq k+1. \label{ineq-teo-sub-criticaldue}
\end{align}
If an index $t$ is missing in the previous set of inequalities we can as before split the main filtration in two subfiltrations $(\filtration)_{\ttI\smallsetminus\{t\}}$ and $(E_{\{t\}},\alpha_{\{t\}})$. Since $R_\ttI(1)+R_\ttI(r)=\sum_{i\in\ttI}\alpha_i$ and we are supposing that the index $t$ does not appear in the inequalities this forces the coefficient of $\alpha_t$ to be one in any expression of the form $R_\ttI(i)+R_\ttI(j_i)$. Therefore $j_k > t$ otherwise the coefficient of $\alpha_t$ in the expression $R_\ttI(k)+R_\ttI(j_k)$ should be two and so there would be an inequality in which $\alpha_t$ would appear with non-zero coefficient which is absurd. In particular $R_{\{t\}}=\alpha_t$ and $P_\ttI+\delta\mu_\ttI = P_{\ttI\smallsetminus\{t\}}+\delta\mu_{\ttI\smallsetminus\{t\}} + P_{\{t\}}+\delta\mu_{\{t\}}$. Moreover, if $k\neq t$, then the maximum is still achieved in $(k,j_k)$, otherwise $k=t$ and $j_s=j_k$ forall $s\geq k$ (otherwise $\alpha_k$ would appear in some inequality of type \eqref{ineq-teo-sub-criticaluno} and \eqref{ineq-teo-sub-criticaldue}) and so the maximum of the filtration indexed by $\ttI\smallsetminus\{t\}$ is realized in $(k+1,j_{k+1})$.\\

The last case we have to consider is when all indexes appear in inequalities \eqref{ineq-teo-sub-criticaluno} and \eqref{ineq-teo-sub-criticaldue}. Note that the set of indexes that appear in inequalities \eqref{ineq-teo-sub-criticaluno}, that we will call $\ttJ$, does not intersect the set of indexes of inequalities \eqref{ineq-teo-sub-criticaldue}, that we will denote with $\ttJ'$; moreover all indexes appearing in inequalities of type \eqref{ineq-teo-sub-criticaluno} (\eqref{ineq-teo-sub-criticaldue} respectively) appear in the first inequality of the same type (the last respectively). If $k$ is such that the set $J$ and $J'$ are both not empty, then an easy calculation shows that $P_\ttI+\delta\mu_\ttI = P_{\ttJ}+\delta\mu_{\ttJ} + P_{\ttJ'}+\delta\mu_{\ttJ'}$. Indeed, in the filtration indexed by $\ttJ$, the maximum is achieved in $(j_k,j_k)$ while the maximum for the filtration indexed by $\ttJ'$ is achieved in $(k,r)$.\\

Finally if $\ttJ=\emptyset$ then $R_\ttI=R_\ttI(1)+R_\ttI(r)$, i.e. $k=1$, and so we have an inequality containing all indexes: $\alpha_1+\dots+\alpha_t\geq\alpha_{t+1}+\dots+\alpha_{r-1}$ for a certain $t\in\ttI$. Then we call $\beta=\alpha_{t+1}+\dots+\alpha_{r-1}$, we write $\beta=\beta_1+\dots+\beta_t$ such that $\alpha_i\geq\beta_i$ for $i=1,\dots,t$ and we consider, for all $i=1,\dots,t$, the weighted filtrations $0\subset E_i\subset E_{t+1}\subset \dots \subset E_{r-1}$ with weights $(\alpha_1,\beta_{i,t+1},\dots,\beta_{i,r-1})$ where $\beta_{i,s}$ are such that $\sum_{s=t+1}^{r-1}\beta_{i,s}=\beta_i$ and $\beta_{i,s}\geq\beta_{i,s'}$ if and only if $\alpha_s\geq\alpha_{s'}$, for any $i=1,\dots,t$. With such choices is easy to see that
\begin{equation*}
 P_\ttI+\delta\mu_\ttI=\sum_{s=1}^{t} P_{\ttJ_s}+\delta\mu_{\ttJ_s},
\end{equation*}
where we denote $\ttJ_s$ the set $\{s,t+1,\dots,r-1\}$. And we are done.
\end{proof}

\begin{remark}
 As a consequence of the proof of Theorem \ref{teo-sub-critical} we have that every critical filtration splits as a certain number of length two critical filtrations and a non-critical one (which obviously can be decomposed as the union of length one filtrations).
\end{remark}

\begin{example}\label{example-teo-sub-critical}
 Let us fix $r=5$ and let $0\subset E_1\subset E_2\subset E_3\subset E_4\subset E$ be a filtration with weight vector $(\alpha_1,\alpha_2,\alpha_3,\alpha_4)$ such that the matrix representing $\map$ with respect the filtration is the following:
\begin{equation*}
 \begin{pmatrix}
0 & 0 & 0 & 0 & 1\\
0 & 0 & 0 & 1 & 1\\
0 & 0 & 1 & 1 & 1\\
0 & 1 & 1 & 1 & 1\\
1 & 1 & 1 & 1 & 1\\
\end{pmatrix}
\end{equation*}
In this case the filtration is critical, in fact, denoting with $\ttI$ the set $\{1,2,3,4\}$, we have that 
\begin{align*}
R_\ttI & =\max\{R_\ttI(1)+R_\ttI(r),R_\ttI(2)+R_\ttI(4),R_\ttI(3)+R_\ttI(3)\}\\
       & =\max\{A=\alpha_1+\alpha_2+\alpha_3+\alpha_4,B=\alpha_2+\alpha_3+2\alpha_4,C=2\alpha_3+2\alpha_4\},
\end{align*}
while $\sum_{i\in\ttI}R_{\{i\}}=\alpha_1+\alpha_2+2\alpha_3+2\alpha_4$.\\

If the maximum is $A$ then we have the following inequalities:
\begin{align*}
 A\geq B & \Rightarrow \alpha_1\geq\alpha_4\\
 A\geq C & \Rightarrow \alpha_1+\alpha_2\geq \alpha_3+\alpha_4.
\end{align*}
So in this case we are in the situation $\ttJ=\emptyset$ considered in the proof of Theorem \ref{teo-sub-critical}; therefore we can find $\alpha_3',\alpha_3'',\alpha_4',\alpha_4''$ such that $\alpha_3'+\alpha_3''=\alpha_3$, $\alpha_4'+\alpha_4''=\alpha_4$, $\alpha_1\geq\alpha_3'+\alpha_4'$ and $\alpha_2\geq\alpha_3''+\alpha_4''$. Let us consider the filtrations $0\subset E_1\subset E_3\subset E_4\subset E$, $0\subset E_2\subset E_3\subset E_4\subset E$ with weights $(\alpha_1,\alpha_3',\alpha_4')$ and $(\alpha_2,\alpha_3'',\alpha_4'')$ respectively. Proceeding as before we split the filtrations $\{134\}$ and $\{234\}$ and we obtain that:
\begin{align*}
 P_\ttI+\delta\mu_\ttI & = P_{\{134\}}+\delta\mu_{\{134\}}+P_{\{234\}}+\delta\mu_{\{234\}}\\
                       & = P_{\{13\}}+\delta\mu_{\{13\}}+P_{\{34\}}+\delta\mu_{\{34\}}+P_{\{2\}}+\delta\mu_{\{2\}}+P_{\{34\}}+\delta\mu_{\{34\}}
\end{align*}
and we are done.\\

If the maximum is $B$ then we have the following inequalities:
\begin{align*}
 B\geq A \Rightarrow \alpha_4\geq\alpha_1\\
 B\geq C \Rightarrow \alpha_2\geq \alpha_3.
\end{align*}
Therefore $\ttJ=\{1,4\}$ and $\ttJ'=\{2,3\}$ are disjoint and an easy calculation shows that $R_\ttJ=R_\ttJ(4)+R_\ttJ(4)=2\alpha_4$ and $R_\ttJ'=R_\ttJ'(2)+R_\ttJ'(5)=\alpha_2+\alpha_3$, therefore $P_\ttI+\delta\mu_\ttI = P_{\{14\}}+\delta\mu_{\{14\}}+P_{\{23\}}+\delta\mu_{\{23\}}$ and we finish.\\

Finally, if the maximum is $C$, $\ttJ'=\emptyset$ and calculations are similar to the case in which $A$ is the maximum.
\end{example}

Thanks to previous results, in order to check the semistability condition, we can focus our attention only on subbundles and critical filtrations of length $2$. More precisely:
\begin{proposition}\label{prop-semistability-reduction}
 Let $(E,\map)$ as before, then the following statements are equivalent:
\begin{enumerate}
 \item $(E,\map)$ is $\delta$-(semi)stable;
 \item For any subbundle $F$ and for any critical weighted filtration $(0\subset E_i\subset E_j\subset E\,,\,\alpha_i,\alpha_j)$ of length two the following inequalities hold:\\
\begin{itemize}
 \item $(d\,r_F-r\,d_F)-\delta(r\,\kk{F}{E}-2\,r_F)(\geq)0$,\\
 \item $P_{\{i,j\}}-\delta\left(\alpha_i r_i+\alpha_j r_j-r \max\{\alpha_i+\alpha_j,2\alpha_j\}\right)(\geq)0$.\\
\end{itemize}
\end{enumerate}
\end{proposition}
\begin{proof}
 The arrow $(1)\Rightarrow(2)$ is trivial. So suppose that $(2)$ holds, then, as noticed before, semistability can be checked only on subbundles and critical filtrations and, thanks to Theorem \ref{teo-sub-critical}, we can consider only critical weighted filtrations of length two. Finally, due to Proposition \ref{prop-mu-critical-filtration}, we get that
\begin{equation*}
 P_{\{i,j\}}+\delta\mu_{\{i,j\}}=P_{\{i,j\}}-\delta\left(\alpha_i r_i+\alpha_j r_j-r \max\{\alpha_i+\alpha_j,2\alpha_j\}\right)
\end{equation*}
and so we are done.
\end{proof}

\begin{lemma}\label{lemma-p-vs-np}
 Let $(E,\map)$ be a quadric bundle such that $P(\efil)+\delta\mu(\efil;\map)\geq0$ for any filtration $\efil$ with weight vector identically $1$. Then $(E,\map)$ is $\delta$-(semi)stable.
\end{lemma}
\begin{proof}
 Clearly $(E,\map)$ is $\kk{}{}$-(semi)stable since weights do not affect the semistability condition for subbundles. Moreover by Theorem \ref{teo-sub-critical} we can check semistability only on critical weighted filtrations of length two. Let $(0\subset E_1\subset E_2 \subset E \,,\, \alpha_1,\alpha_2 )$ be such a filtration. We want to show that
\begin{equation*}
 P+\delta\mu=\alpha_1 c_1 +\alpha_2 c_2 + r\delta\max\{\alpha_1+\alpha_2,2\alpha_2\} (\geq) 0.
\end{equation*}
If the maximum is $\alpha_1+\alpha_2$, then $\alpha_1\geq\alpha_2$ and the previous inequality becomes:
\begin{equation*}
  \alpha_1 c_1 +\alpha_2 c_2 + r\delta(\alpha_1+\alpha_2) = \alpha_2(c_1+c_2+2r\delta) + (\alpha_1-\alpha_2)(c_1+r\delta)(\geq)0
\end{equation*}
 where the last inequality holds since by hypothesis $c_1+c_2+2r\delta$ and $c_1+r\delta$ are non-negative.\\
Otherwise, if the maximum is $2\alpha_2$, then $\alpha_1\leq\alpha_2$ so
\begin{equation*}
\alpha_1 c_1 +\alpha_2 c_2 + r\delta(2\alpha_2) = \alpha_1(c_1+c_2+2r\delta) + (\alpha_2-\alpha_1)(c_2+2r\delta)\geq0,
\end{equation*}
and $(E,\map)$ is semistable.
\end{proof}

\begin{definition}\label{def-semistab-2}
A quadric bundle $(E,\map)$ is $\delta$-(semi)stable if the following conditions hold: 
\begin{enumerate}
\item If $F$ is a proper subbundle of $E$, then
\begin{equation*}
\frac{\deg(F)-\delta\kk{F}{E}}{\rk{F}} (\leq) \frac{\deg(E) -
2\delta }{r}.
\end{equation*}
\item If $0\subset E_i\subset E_j \subset E$ is a critical filtration, then
\begin{equation*}
(r_i+r_j)d-r(d_i+d_j)-2\delta(r_i+r_j-r) (\geq) 0.
\end{equation*}
\end{enumerate}
\end{definition}

\begin{corollary}\label{cor-equiv-ss-conditions}
 Definition \ref{def-semistability} and Definition \ref{def-semistab-2} are equivalent
\end{corollary}
\begin{proof}
 Follows directly from Theorem \ref{teo-sub-critical} and Lemma \ref{lemma-p-vs-np}.
\end{proof}

Let us observe that if $\rk{E}=3$ then the semistability condition is equivalent to the one given by Gomez and Sols in \cite{GS}, while, if $\rk{E}=2$, it is equivalent to the one given by Gothen and Oliveira in \cite{GO}.\\

From now on we will use Definition \ref{def-semistab-2} as definition of semistability and therefore we will always consider just weighted filtrations with weight vector identically one.\\

\subsection{Proof of Theorem \ref{teo-main}}
Let $(E,\varphi)$ be a decorated vector bundle of type $(2,b,c,\linebu)$, replacing $\linebu'=(\det E)^{\otimes c}\otimes\linebu$ we can assume, without loss of generality, $c=0$. Moreover, fixed a weighted filtration $(\filtration)$, one can compute the matrix $M_{\ttI}(\efil,\varphi)$, which, by definition, will be a symmetric matrix although $\varphi$ is not symmetric. So the proof of Theorem \ref{teo-sub-critical} still holds also in this case. Similarly one can easily show that also all the results in Section \ref{section-quadric} are still true for decorated bundles in general and so we get the thesis.

\subsection{Maximal destabilizing subbundle}


Let $(E,\map)$ be a strictly semistable quadric bundle, then there exists a subbundle $F$ or a filtration $0\subset H_1\subset H_2\subset E$ such that at least one of the following equalities holds (i.e. we suppose that there exists a bundle or a filtration which ``de-semistabilizes'' the quadric bundle):
\begin{enumerate}
 \item $r(\deg(F)-\delta\kk{F}{E}) = r_F(\deg(E)-2\delta)$;
 \item $(r_1+r_2)d-r(d_1+d_2)+2\delta(r-r_1-r_2) = 0$.
\end{enumerate}

 We prove that in case $(1)$ the quadric bundle $(F,\map)$ is $\delta$-semistable if $\delta$ is small enough. Let $F'\subset F$ be a subbundle, we want to show that
\begin{equation}\label{eq-jh1}
 \mu(F') - \frac{\delta\,\kk{F'}{F}}{r'}\leq \mu(F) - \frac{\delta\,\kk{F}{F}}{r_F}.
\end{equation}
Since $(E,\map)$ is semistable, $F$ de-semistabilizes $E$ and $F'$ is a subbundle of $E$, we have that
\begin{equation}\label{eq-jh2}
\mu(F') - \frac{\delta\,\kk{F'}{E}}{r'}\leq \mu(E) - \frac{2\delta}{r} = \mu(F) - \frac{\delta\,\kk{F}{E}}{r_F} 
\end{equation}
Since $\kk{F'}{F}\leq\kk{F'}{E},\kk{F}{F}\leq\kk{F}{E}$, if $\kk{F}{E}=0$ then $\kk{F'}{F}=\kk{F}{F}=\kk{F'}{E}=0$ and \eqref{eq-jh1} and \eqref{eq-jh2} are equivalent. If $\kk{F}{E}=1$, $\kk{F}{F}=\kk{F'}{F}=0$ then $\kk{F'}{E}$ could be equal to $0$ or $1$. If $\kk{F'}{E}=0$ then inequality \eqref{eq-jh2} implies \eqref{eq-jh1}, otherwise, if $\kk{F'}{E}=1$, we need to assume $\delta<1/r$ to get the thesis. Lastly if $\kk{F}{E}=2$ then $\kk{F}{F}=2$. If $\kk{F'}{F}=\kk{F'}{E}$ we are done, otherwise $\kk{F'}{F}=0$ and $\kk{F'}{E}=1$; this implies that the filtration $0\subset F'\subset F \subset E$ is critical and so
\begin{equation*}
 (r' + r_F)\,d - r\,(d' + d_F) + 2\delta\,(r - r' - r_F)\geq0
\end{equation*}
but $d r_F - 2\delta r_F = r d_F - 2\delta r$ and substituting in the previous inequality we get
\begin{equation*}
\mu(F')\leq\mu(E)-\frac{2\delta}{r}=\mu(F)-\frac{\delta\,\kk{F}{F}}{r_F}.
\end{equation*}
Finally let $0\subset F'\subset F''\subset F$ be a critical filtration, then $0\subset F'\subset F''\subset E$ is a critical filtration of $E$. Note that $\kk{F}{E}=2$ otherwise $F$ has no critical filtrations. As before, since $(E,\map)$ is semistable, we have that
\begin{equation*}
 (r'+r'')\,d - r\,(d'+d'')+2\delta\,(r-r'-r'')\geq 0;
\end{equation*}
by replacing $r_F d = r d_F + 2\delta r_F - 2\delta r$ in the previous inequality we get the desired inequality and so $(F,\map)$ is semistable as a quadric bundle.\\

\section{Orthogonal bundles}\label{orthogonalbundles}

An orthogonal bundle is a vector bundle associated to a principal bundle with (complex) orthogonal structure group. Equivalently, it is a quadric bundle $(E,\map)$ with $\linebu=\ofascio_X$, such that the bilinear form $\map:\Sym ^2
E \to \ofascio_X$ induces an isomorphism $\map:E \to E^\vee$. In this case $\map$ gives a smooth quadric $C_x$ for each point $x\in X$. Note that the isomorphism $\map:E \to E^\vee$ forces the degree of $E$ to be zero.

There is a notion of stability for orthogonal bundles (see \cite{Ramanan} Ramanan): an orthogonal bundle $E$ is (semi)stable if and only if for every proper isotropic subbundle $F$ (i.e. $\kk{F}{E}\leq 1$), 
\begin{equation*}
\deg(F)\,(\leq)\,0=\deg (E).
\end{equation*}

We will prove that an orthogonal bundle is (semi)stable if and only if it is $\delta$-(semi)stable as a 
quadric bundle. We start with the following useful result:

\begin{lemma}\label{lemmaorthogonal}
Let $(E,\map)$ be an orthogonal bundle, and let $F$ be a proper vector subbundle of $E$. Then
\begin{enumerate}
\item There is an exact sequence
\begin{equation*}
0 \to F^\perp \to E \to F^\vee \to 0,
\end{equation*}
$\deg(F)=\deg(F^\perp)$ and $\rk{F^{\perp}}+\rk{F}=r$.

\item $\kk{F}{E}\geq 1$

\item If $F$ is isotropic (i.e $\kk{F}{E}\leq 1$), then 
\begin{equation*}1\leq\rk{F}\leq\lfloor \frac{r}{2}\rfloor,\end{equation*}
where $\lfloor x\rfloor$ is the largest integer less than or equal to $x$ and $F^\perp$ denote as usual the orthogonal of $F$ with respect to the nondegenerate bilinear form $\map$.
\end{enumerate}
\end{lemma}
\begin{proof}
For the proof of $(1)$ see Gomez and Sols \cite{GS}. $(2)$ and $(3)$ depend on the fact that we are assuming the matrix non-degenerate.
\end{proof}

\begin{lemma}\label{lem-filtr-critica}
 Let $(E,\map)$ be a quadric bundle such that $\map$ is non-degenerate. Let $F$ be a proper isotropic vector subbundle of $E$. Then the filtration $0\subset F\subsetneq F^{\perp}\subset E$ is critical.
\end{lemma}
\begin{proof}
 Let $F'$ be the maximal isotropic subbundle of $E$ containing $F$ and let $r'$ denote its rank. Let $A=(a_{ij})$ the matrix representing $\map$ with respect to a basis of $E$ subordinated to $F'$. Then $a_{r'r'}=0$, $a_{r'+1r'+1}=1$ and
\begin{equation*}
A=
 \begin{pmatrix}
0 &        &  \phantom{0}\;\vline &   &        &  \\
  & \ddots &  \phantom{0}\;\vline &   &   B    &  \\
  &        &            0\;\vline &   &        &  \\ \hline
  &        &  \phantom{0}\;\vline & 1 &        &  \\
  &   B^t  &  \phantom{0}\;\vline &   & \ddots &  \\
  &        &  \phantom{0}\;\vline &   &        & 1\\
\end{pmatrix}
\end{equation*}
The matrix $B$ is a $r'\times(r-r')$-matrix, every row contains at least a $1$ and two different rows must be independent. This forces the matrix $B$ to be of the following form:
\begin{equation*}
 \begin{pmatrix}
0      & \dots  & 0 & \star\\
\vdots & \iddots & \iddots & \bullet\\
0      & \iddots & \iddots & \vdots\\
\star  & \bullet & \ldots & \bullet
\end{pmatrix}
\text{ if } r \text{ is even, or }
 \begin{pmatrix}
0      & 0      & \dots  & 0 & \star\\
\vdots & \vdots & \iddots & \iddots & \bullet\\
\vdots & 0      & \iddots & \iddots & \vdots\\
0      &\star  & \bullet & \ldots & \bullet
\end{pmatrix}
\text{ if } r \text{ is odd,}
\end{equation*}
where ``$\star$'' is any non-zero complex number while ``$\bullet$'' denote any complex number.\\
 In both cases, since $F'\subseteq F^\perp$, we have that
\begin{itemize}
 \item $\rest{\map}{FF}=\rest{\map}{FF^{\perp}}=0$
\item $\rest{\map}{FE},\rest{\map}{F^{\perp}E},\rest{\map}{EE},\rest{\map}{F^{\perp}F^{\perp}}\neq0$
\end{itemize}
and the filtration $0\subset F\subset F^{\perp}\subset E$ is critical.
\end{proof}

\begin{remark}\label{rem-perp}
 With the same notation as before, if $F=F^{\perp}$, then $2 r_F=r$ (in particular $r$ is even); moreover, if $(E,\map)$ is a $\delta$-semistable quadric bundle, condition $(1)$ of Definition \ref{def-semistab-2} tells us that
\begin{equation*}
 d r_F - r d_F + \delta \left( 2 r_F - r \right) \geq 0
\end{equation*}
and so $\mu(F)\leq\mu(E)$.
\end{remark}

\begin{theorem}\label{teo-orthogonal-semistability}
An orthogonal bundle is (semi)stable if and only if it is $\delta$-(semi)stable as a quadric bundle.
\end{theorem}
\begin{proof}
We prove the assertion for semistability, being the proof for stability very similar.\\

Let $(E,\map)$ be a quadric bundle such that $\map:Sym^2 E\to \ofascio_X$ is non-degenerate, and assume that it is $\delta$-(semi)stable. Let $F$ be an isotropic vector subbundle. If $F\neq F^\perp$, the filtration $0\subset F \subsetneq F^\perp \subset E$ is critical by Lemma \ref{lem-filtr-critica}. Then the semistability condition with weights identically $1$ tells us that
\begin{equation*}
 (r_F+r_{F^\perp})\, d - r\, (d_F + d_{F^\perp})+2\delta(r-r_F-r_{F^\perp}) \geq 0.
\end{equation*}
By the first point of Lemma \ref{lemmaorthogonal} $d_F = d_{F^\perp}$ and $r=r_F+r_{F^\perp}$. Since $E$ is an orthogonal bundle $\deg(E)=0$ and the above inequality tells us that $\deg(F)\leq 0$. If $F=F^{\perp}$, by Remark \ref{rem-perp} we still have $\mu(F)\leq 0$ which proves that $E$ is semistable as an orthogonal bundle.\\

Conversely, let $E$ be a semistable orthogonal bundle. Let $F$ be any vector subbundle. Following Ramanan (see \cite{Ramanan}) let $N=F \cap F^\perp$, and let $N'$ be the vector subbundle generated by $N$. We have an exact sequence
\begin{equation}\label{orthoshort}
0 \to N' \to F \oplus F^\perp \to M' \to 0
\end{equation}
where $M'$ is the subbundle of $E$ generated by $F + F^\perp$. We have $M'=(N')^\perp$. 
If $N'=0$, then $E=F \oplus F^\perp$, $\kk{F}{E}=2$, and $\deg(F)=\deg(E)=0$ (Lemma \ref{lemmaorthogonal}). Then 
\begin{equation*}
\frac{\deg(F) - \kk{F}{E}\delta}{\rk{F}}=\frac{-2\delta}{\rk{F}}<
\frac{-2\delta}{r}=\frac{\deg(E) - 2\delta}{r}.
\end{equation*}

If $N'\neq 0$, $\deg(F)=\deg(N')$ (by Lemma \ref{lemmaorthogonal} and the exact sequence \eqref{orthoshort}), and then $\deg(F)\leq 0$ (because $E$ is a semistable orthogonal bundle and $N'$ is isotropic). Recalling that if $\kk{F}{E}\leq 1$ then $1\leq\rk{F}\leq\lfloor\frac{r}{2}\rfloor$, if $\kk{F}{E}=2$ there are no conditions on the rank of $F$ but in any case we have:
\begin{equation*}
\frac {\deg(F) - \delta\,\kk{F}{E}}{\rk{F}} \leq -\frac{
\delta\,\kk{F}{E}}{\rk{F}} \leq -\frac{2\delta}{r}=\frac{\deg(E) - 2\delta}{r}.
\end{equation*}
Now let $0\subset E_i \subset E_j \subset E$ be a critical filtration. Then $E_i$ is isotropic and $E_j \subseteq E_i ^\perp$(otherwise the filtration would not be critical). Therefore $\rk{E_j}\leq\rk{E_i^{\perp}}$ and thanks to previous calculations $\deg(E_i),\deg(E_j)\leq 0$. Finally we have
\begin{equation*}
P_{\{i,j\}}+\delta\mu_{\{i,j\}}=-r(\deg(E_i)+\deg(E_j))+2\delta(r-\rk{E_i}-\rk{E_j})\geq 0,
\end{equation*}
and so by Definition \ref{def-semistab-2} $(E,\map)$ is $\delta$-semistable as a quadric bundle.
\end{proof}

\begin{remark}
 In \textit{``Orthogonal and spin bundles over hyperelliptic curves''} Ramanan shows that an orthogonal bundle is semistable if and only if it is semistable as a vector bundle (Proposition 4.2). So as a corollary of Theorem \ref{teo-orthogonal-semistability} we obtain that a non-degenerate quadric bundle of degree zero $(E,\map)$ is semistable if and only if $E$ is a semistable vector bundle.
\end{remark}

\subsection{Generalized orthogonal bundles}

We will call \textit{generalized orthogonal bundle} a quadric bundle $(E,\map)$ such that the morphism $\map\colon Sym^2E\to\linebu$ induces an isomorphism $E\to E^\vee\otimes\linebu$. This isomorphism connect the degree $d$ of $E$ with the degree $n$ of $\linebu$, in fact one has that $d=-d+rn$ and so $n=2\mu(E)$.

For these objects a similar result to Lemma \ref{lemmaorthogonal} holds:

\begin{lemma}
 Let $(E,\map)$ be a generalized orthogonal bundle, and let $F$ be a proper vector subbundle of $E$. Then there is an exact sequence
\begin{equation*}
0 \to F^\perp \to E \to F^\vee\otimes\linebu \to 0,
\end{equation*}
so $\rk{F^{\perp}}+\rk{F}=r$ and
\begin{equation*}
\deg(F)=\deg(F^\perp)-d\left(1-\frac{2r_F}{r}\right).
\end{equation*}
\end{lemma}

Thanks to the previous lemma, one can easily prove that Theorem \ref{teo-orthogonal-semistability} holds also for generalized orthogonal bundles.

\section{Generalization: Nodal curves}\label{section-nodal}

Let $X$ be a nodal curve with a simple node $x_0$, let $\nu\colon Y\to X$ be the normalization map and finally let $\{y_1,y_2\}=\nu^{-1}(x_0)$. Fix a vector space $R$ of dimension $\rk{E}$. A \textit{parabolic vector bundle} over $Y$ with support $y_1,y_2$ is a pair $(E,q)$ where $E$ is a vector bundle over $Y$, while $q\colon E_{y_1}\oplus E_{y_2}\to R$ is a surjective morphism of vector spaces. Let $(E,q)$ be a parabolic vector bundle over $Y$, and define
\begin{equation}
\degpar(E)\doteqdot\deg(E)-\dim(q(E_{y_1}\oplus E_{y_2})).
\end{equation}
We say that $(E,q)$ is \textit{(semi)stable} if for any subbundle $F\subset E$ the following inequality holds:
\begin{equation}
\frac{\degpar(F)}{\rk{F}}(\leq)\frac{\degpar(E)}{\rk{E}},
\end{equation}
where the parabolic structure over $F$ is induced by the parabolic structure of $E$.\\

Let $(E,\varphi,q)$ be a parabolic decorated bundle, i.e., the datum of a decorated vector bundle $(E,\varphi)$ with a parabolic structure $q$ making $(E,q)$ a parabolic vector bundle. A semistability condition for such objects is obtained by replacing $\deg$ in Definition \ref{def-semistability} by $\degpar$. Using this semistability condition one can show that there is a one to one correspondence between semistable decorated torsion free sheaves $\efascio$ over $X$ and semistable parabolic decorated vector bundles over $Y$ satisfying some ``descent'' conditions (\cite{Sch1} and \cite{LP}).\\

A straightforward calculation shows that Theorem \ref{teo-main} holds also for parabolic decorated bundles and so, thanks to the above correspondence, it also holds for decorated bundles over nodal curves.

\bigskip
\textbf{Acknowledgments.}
We would like to thank Professor Ugo Bruzzo for discussions and comments, Professor Ignasi Mundet I Riera and Professor Alexander H.W. Schmitt for useful explanations and interesting remarks.

\bigskip
\end{document}